\documentclass[11pt]{amsart}

\usepackage[a4paper,hmargin=3cm,vmargin=3cm]{geometry}
\usepackage{amsfonts,amssymb,amscd,amstext,hyperref}
\usepackage{graphicx}
\usepackage{color}
\usepackage[dvips]{epsfig}
\usepackage[latin1]{inputenc}

\usepackage{fancyhdr}
\usepackage{times}

\usepackage{enumerate}
\usepackage{titlesec}
\usepackage{mathrsfs}
\usepackage{stmaryrd}

\def\R{\mathbb{R}}
\def\B{\mathbb{B}}
\def\N{\mathbb{N}}
\def\Z{\mathbb{Z}}

\def\Q{\mathbb{Q}}

\newcommand{\ben}{\begin{enumerate}}
\newcommand{\bit}{\begin{itemize}}
\newcommand{\een}{\end{enumerate}}
\newcommand{\eit}{\end{itemize}}

\newcommand{\ed}{\end{document}}

\def\cC{\mathcal{C}}

\def\cR{\mathcal{R}}

\def\cW{\mathcal{W}}

\def\cL{\mathcal{L}}
\def\cM{\mathcal{M}}

\def\cO{\mathcal{O}}

\let\8=\infty \let\0=\emptyset  
\let\hat=\widehat

\let\landa=\lambda
\let\alfa=\alpha

\let\parc=\partial

\def\ep{\varepsilon}

\def\landa{\lambda}

\def\flecha{\rightarrow}
\def\esiz{\langle}
\def\esde{\rangle}

\def\S{\Sigma}

\def\cte.{\mathop{\rm cte.}\nolimits}

\def\cosh{\mathop{\rm cosh }\nolimits}

\def\N{\mathbb{N}}

\def\B{\mathbb{B}}
\def\Q{\mathbb{Q}}
\def\R{\mathbb{R}}
\def\Z{\mathbb{Z}}

\def\S{\mathbb{S}}

\newfont{\bb}{msbm10 at 12pt}

\titleformat{\section}
{\filcenter\bfseries\large} {\thesection{.}}{0.2cm}{}
\titleformat{\subsection}[runin]
{\bfseries} {\thesubsection{.}}{0.15cm}{}[.]
\titleformat{\subsubsection}[runin]
{\em}{\thesubsubsection{.}}{0.15cm}{}[.]

\usepackage[up,bf]{caption}

\newtheorem{theorem}{Theorem}[section]
\newtheorem{lemma}[theorem]{Lemma}
\newtheorem{proposition}[theorem]{Proposition}

\newtheorem{remark}[theorem]{Remark}
\newtheorem{corollary}[theorem]{Corollary}
\newtheorem{definition}[theorem]{Definition}

\newtheorem{claim}[theorem]{Claim}

\theoremstyle{definition}

\numberwithin{equation}{section}
\numberwithin{figure}{section}

\usepackage{color}

\begin{document}
\fancyhead[LO]{Free boundary CMC annuli}
\fancyhead[RE]{Alberto Cerezo, Isabel Fernández, Pablo Mira}
\fancyhead[RO,LE]{\thepage}

\thispagestyle{empty}

\begin{center}
{\bf \LARGE Annular solutions to the partitioning problem in a ball} 

\vspace*{5mm}

\hspace{0.2cm} {\Large Alberto Cerezo, Isabel Fernández and Pablo Mira}
\end{center}

%



\footnote[0]{
\mbox{} \hspace{0.25cm} \emph{Keywords}: constant mean curvature, free boundary, capillary surfaces, spherical curvature lines, nodoid, partitioning problem}

\vspace*{7mm}

\begin{quote}
{\small
\noindent {\bf Abstract}\hspace*{0.1cm}
For any $n\in \mathbb{N}$, $n\geq 2$, we construct a real analytic, one-parameter family of compact embedded CMC annuli with free boundary in the unit ball $\mathbb{B}^3$ of $\mathbb{R}^3$ with a prismatic symmetry group of order $4n$. These examples give a negative answer to the uniqueness problem by Nitsche and Wente of whether any annular solution to the partitioning problem in the ball should be rotational.


\vspace*{0.1cm}

}
\end{quote}

\section{Introduction}


The classical \emph{partitioning problem} in the unit ball $\B^3$ of $\R^3$ asks to find, among all surfaces that divide $\B^3$ into two pieces of prescribed volumes, those which are critical points for the area functional. By variational and regularity theory, the solutions to this problem are known to be compact embedded surfaces $\Sigma$ of constant mean curvature (CMC) contained in $\B^3$, that intersect $\parc \B^3$ orthogonally along their boundary $\parc \Sigma$, see e.g. \cite{GHN,H}. In other words, they are embedded \emph{free boundary CMC surfaces} in $\B^3$ \cite{Fi,Nit,Ros,Str}.

It is well known that the solutions to the isoperimetric problem in $\B^3$ under this prescribed volume constraint are either flat disks or spherical caps (\cite{BS,BM}, see also the survey \cite{Ros}). More generally, it follows by the results of Ros-Vergasta \cite{RV} and Nunes \cite{Nu} that any \emph{stable} solution to the partitioning problem in $\B^3$ is a flat disk or a spherical cap; see also \cite{WX}. Along the same lines, Nitsche \cite{Nit} proved in 1985 the topological uniqueness of such simply connected examples, in the following form: \emph{any free boundary CMC surface in $\B^3$ with the topology of a disk is either a flat disk or a spherical cap.} See also Ros and Souam \cite{RS}.

The topological uniqueness problem for \emph{annular} solutions to the partitioning problem was formulated by Nitsche in 1985. In \cite{Nit}, Nitsche indicated that there exist certain embedded doubly connected rotational CMC surfaces that have free boundary in $\B^3$. For the case of minimal surfaces, i.e. when the mean curvature $H$ is zero, this surface is the well-known \emph{critical catenoid}, while for other values of $H\neq 0$, they are compact pieces of Delaunay surfaces (unduloids or nodoids). Nitsche claimed without proof in \cite{Nit} that any free boundary CMC annulus in $\B^3$ should be rotational, and hence one of these examples. Later on, Wente \cite{W2} revisited this problem and asked whether any annular solution to the partitioning problem in $\B^3$ (i.e. any embedded free boundary CMC annulus in $\B^3$) should be rotational. 

Wente obtained in \cite{W2} examples of non-embedded free boundary CMC annuli in $\B^3$ with very large mean curvature. For the $H=0$ case, Fernández, Hauswirth and Mira have recently constructed free boundary minimal annuli in $\B^3$, see \cite{FHM}. These examples in \cite{FHM,W2} give non-embedded counterexamples to Nitsche's claim in \cite{Nit}. In \cite{FHM} one can also find examples of embedded \emph{capillary} minimal annuli in $\B^3$ that intersect $\parc \B^3$ at a constant angle $\theta\neq \pi/2$. This still leaves unanswered the important \emph{critical catenoid conjecture}, according to which the critical catenoid should be the only embedded free boundary minimal annulus in $\B^3$; see Fraser and Li \cite{FL}, and also \cite{L,FHM} for a more updated discussion.

Our aim in this paper is to construct a large family of embedded free boundary CMC annuli in $\B^3$ that are not Delaunay surfaces. This gives a negative answer to the 1995 uniqueness problem by Wente in \cite{W2}, and produces the first known annular solutions to the partitioning problem in a ball that are not rotational. It also shows that Nitsche's topological uniqueness claim in \cite{Nit} is not true even in the embedded case.

\begin{theorem}\label{th:main}
For any $n\in \N$, $n>1$, there exists a real analytic family $\{\mathbb{A}_n(\mu) : \mu \in [0,\ep)\}$ of compact CMC annuli in $\R^3$ such that:
\begin{enumerate}
\item
$\mathbb{A}_n(\mu)$ is contained in the closed unit ball $\B^3$, with $\parc \mathbb{A}_n(\mu)\subset \parc \B^3$.
 \item
$\mathbb{A}_n(\mu)$ is embedded.
\item
$\mathbb{A}_n(\mu)$ intersects $\parc \B^3$ orthogonally along its boundary, i.e., $\mathbb{A}_n(\mu)$ is free boundary.
\item
Each $\mathbb{A}_n(\mu)$ has a family of spherical curvature lines. 
 \item
Each $\mathbb{A}_n(\mu)$ is symmetric with respect to the $x_3=0$ plane, and with respect to $n$ equiangular vertical planes containing the $x_3$-axis.  \item
$\mathbb{A}_n(0)$ is a compact piece of a free boundary nodoid in $\B^3$. If $\mu>0$, then $\mathbb{A}_n(\mu)$ is not rotational, and has a prismatic symmetry group of order $4n$. In this way, the symmetry group of $\mathbb{A}_n(\mu)$, $\mu>0$, is isomorphic to $D_n\times \Z_2$.
\end{enumerate}
\end{theorem}

We remark that these embedded CMC annuli $\mathbb{A}_n(\mu)$ are not minimal and that, for each $n>1$, the analytic map $\mu \mapsto H(\mu)$ that assigns to each $\mu$ the (constant) mean curvature of $\mathbb{A}_n(\mu)$ is not constant. Moreover, the mean curvature vector of $\mathbb{A}_n(\mu)$ points \emph{outwards} along the horizontal planar geodesic $\mathbb{A}_n(\mu)\cap \{x_3=0\}$, and for $\mu>0$ small enough, $\mathbb{A}_n(\mu)$ has negative Gaussian curvature. The constant mean curvatures $H_n$ of the embedded free boundary nodoids $\mathbb{A}_n(0)$ in $\B^3$ of Theorem \ref{th:main} satisfy that $H_n\to \8$ as $n\to \8$.

The proof of Theorem \ref{th:main} is related to the discovery by Wente \cite{W0} of immersed CMC tori in $\R^3$ in the 1980s. After \cite{W0}, Abresch \cite{Ab} and Walter \cite{Wa1} realized that one could construct these CMC tori in $\R^3$ having one family of planar curvature lines; in that case, the \emph{other} family of curvature lines of the CMC torus must be \emph{spherical}, that is, each element of such family must lie in some sphere or plane of $\R^3$. Later on, motivated by classical works of geometers of the 19th century, Wente studied in \cite{W} the class of CMC surfaces in $\R^3$ with a family of spherical curvature lines, producing new examples of CMC tori. In \cite{W2}, Wente used the same idea to construct immersed, non-embedded free boundary annuli in $\B^3$ with very large constant mean curvature. Let us remark that, in all these constructions, the family of spherical, non-planar curvature lines is the one associated to the \emph{largest} principal curvature of the torus $\Sigma$, when we assume that $\Sigma$ is oriented so that $H>0$.

In the present work we follow a similar geometric ansatz, and construct embedded free boundary CMC annuli in $\B^3$ that have one family of spherical curvature lines. However, our construction is carried out in a region of the moduli space of all CMC surfaces with spherical curvature lines that remained unexplored in the previous works by Wente, Abresch and Walter. For instance, one of the aspects where we deviate from these works is that we chose the spherical curvature lines to correspond to the \emph{smallest} principal curvature of the CMC surface $\Sigma$. This choice is key for the process to work; indeed, if one choses the \emph{largest} principal curvature as the spherical one, one can still create free boundary CMC annuli in $\B^3$ but our study suggests that they will never be embedded. Let us also remark that despite these similarities, our proof of Theorem \ref{th:main} in order to reach embeddedness is different in nature to the construction of Wente in \cite{W2}.

The \emph{critical catenoid conjecture} mentioned above is usually conceived as the natural free boundary version in the ball of the Lawson conjecture, which claims that the Clifford torus is the only embedded minimal torus in the unit sphere $\S^3$. The Lawson conjecture was answered affirmatively by Brendle in \cite{B}. The proof by Brendle was extended by Andrews and Li \cite{AL} to the CMC case, proving that \emph{any embedded CMC torus in $\S^3$ is rotational}. In this sense, our Theorem \ref{th:main} shows that the natural free boundary version in $\B^3$ of the Brendle-Andrews-Li uniqueness theorem in $\S^3$ is not true.


A brief outline of the paper is as follows. In Section \ref{sec:2} we describe a general method to produce CMC surfaces $\Sigma$ in $\R^3$ with a family of spherical curvature lines corresponding to the \emph{smallest} principal curvature $\kappa_2$ of $\Sigma$. The procedure essentially comes from Wente \cite{W}, but we need to make a different discussion here. In this construction, we end up with three free parameters $(\alfa,\beta,\gamma)$. In Section \ref{sec:3} we show how certain choices of these parameters give rise to immersed CMC annuli $\Sigma$ with a discrete symmetry group. In Section \ref{sec:4} we find a suitable region in the parameter space $(\alfa,\beta,\gamma)$ where we can control the condition that $\Sigma$ intersects orthogonally a sphere along each boundary curve. Finally, in Section \ref{sec:5} we complete the proof of Theorem \ref{th:main}, by controlling the embeddedness of the surface, and the property that both boundary curves intersect orthogonally the same sphere. In Section \ref{sec:discussion} we discuss some open problems. 


\section{A family of CMC surfaces with spherical curvature lines}\label{sec:2}

Consider the parameter domain 
\begin{equation}\label{domain}
\cO :=\{(\alfa,\beta,\gamma)\in \R^3 : \alfa\geq 1, \beta\geq 1, \gamma\geq 1\}.
\end{equation}
Define in terms of $(\alfa,\beta,\gamma)\in \cO$ the polynomial of degree four
 \begin{equation}\label{def:pab}
p(x)= -\left(x-\frac{\alfa}{\gamma}\right)\left(x-\frac{1}{\alfa \gamma}\right)\left(x+\beta\gamma\right)\left(x+\frac{\gamma}{\beta}\right).
 \end{equation}
Note that $p(0)=-1$ and that:
\begin{enumerate}
\item
If $\alfa\neq 1$ and $\beta \neq 1$, then $p(x)$ has two positive roots and two negative roots.
 \item
If $\alfa=1$ (resp. $\beta =1$), the two positive (resp. negative) roots of $p(x)$ collapse into a single double root at $x=1/\gamma$ (resp. $x= -\gamma$).
\end{enumerate}

Define next in terms of $(\alfa,\beta,\gamma)\in \cO$ the constants 
\begin{equation}\label{const2}
A=\frac{1}{2}(\alfa+ \alfa^{-1}), \hspace{0.5cm} B=\frac{1}{2} (\beta+\beta^{-1}), \hspace{0.5cm} C= \frac{1}{2}(\gamma-\gamma^{-1}).
\end{equation}
%
%
Let $(y(u),z(u)):\R\flecha \R^2$ be the unique solution to the differential system 
\begin{equation}\label{system1}
\left\{\def\arraystretch{1.3} \begin{array}{lll} y'' & = & (\hat{a}-1) y - 2 y (y^2-z^2), \\ z'' & = & \hat{a}z - 2 z (y^2-z^2),\end{array} \right.
\end{equation}
with the initial conditions
\begin{equation}\label{inicondi}
y(0)=0, \hspace{0.2cm} z(0)=0, \hspace{0.3cm} y'(0)= \frac{(A+B)C}{2}, \hspace{0.3cm} z'(0)= \frac{(B-A)\sqrt{C^2+1}}{2},
\end{equation}
where $\hat a:= 1-AB+C^2$. 

We will next create a solution $\omega(u,v)$ to the sinh-Gordon equation from $(y(u),z(u))$. To start, let $\varrho_0 := 1/(\alfa\gamma)>0$ denote the smallest positive root of $p(x)$, and define $\omega$ along $u=0$ as $e^{\omega(0,v)}=x(v)$, where $x(v)$ is the unique non-constant analytic solution to the ODE 
\begin{equation}\label{omini}
4 x'(v)^2= p(x(v)),
\end{equation}
with the initial condition $x(0)=  \varrho_0.$ Alternatively, $x(v)$ can be defined as the unique solution to $8x''(v)=p'(x(v))$ with initial conditions $x(0)=\varrho_0$, $x'(0)=0$. Observe that $x(v)$ is a
 periodic function on $\R$ that takes values in $[\varrho_0,\varrho_1]$, where $\varrho_1:=\alfa/\gamma$ is the other positive root of $p(x)$. 
Next, define $\omega(u,v)$ for any $(u,v)\in \R^2$ from these initial values $\omega(0,v)$ by imposing the relation
\begin{equation}\label{omu}
\omega_u = y(u) \cosh \omega + z(u)\sinh \omega,
\end{equation}
which can be seen as a Riccati ODE.
This process clearly determines a unique real analytic function $\omega(u,v)$, and we have:

\begin{lemma}\label{lem:gordon}
$\omega$ is a solution to the sinh-Gordon equation
\begin{equation}\label{sinhg}
\Delta \omega +\sinh \omega  \cosh \omega =0.
\end{equation}
\end{lemma}
\begin{proof}
Let $X(u,v):=e^{\omega(u,v)}$, and note that, by \eqref{omu},
\begin{equation}\label{omux}
2X_u =  y(u) (X^2 +1)+ z(u) (X^2-1).
\end{equation}
Define in terms of $X,y(u),z(u)$ the function $\phi(u,v)$ given by
\begin{equation}\label{pfi}
\phi:= -(1+(y+z)^2)X^4 -4 (y'+z') X^3 + 6 \hat{\gamma}  X^2+ 4 (y'-z') X - (1+(y-z)^2),
\end{equation} 
where here $\hat{\gamma}=\hat{\gamma}(u)$ is given by $6\hat{\gamma}:= 6(y^2-z^2)-4(\hat{a}-1/2)$.  Using \eqref{inicondi},  a computation shows that 
\begin{equation}\label{fiu}
\phi(0,v)=p(x(v)).
\end{equation} 
Let us check that $X$ satisfies the differential equation 
\begin{equation}\label{xv}
4X_v^2=\phi. 
\end{equation}
First, a direct  computation from \eqref{system1} and \eqref{omux} gives
$$ \phi_u= 2(y+z)  X \phi $$
On the other hand, differentiating \eqref{omux} with respect to $v$ we get
$$ X_{uv}= (y+z) X X_v.$$ 
Comparing these two expressions and integrating with respect to $u$, we obtain 
$X_v^2 = F(v) \phi$
 for some function $F(v)$.  Evaluating at $u=0$ and using  \eqref{omini}, \eqref{fiu}, we deduce that $F(v)=1/4$, and so \eqref{xv} holds, as desired. If we now differentiate \eqref{xv} with respect to $v$, as well as \eqref{omux} with respect to $u$, a computation using \eqref{pfi}, \eqref{xv} shows that $$\Delta X= \frac{X}{4}(-X^2+1/X^2) + \frac{X_u^2 +X_v^2}{X}.$$ Noting that $X=e^{\omega}$, we obtain \eqref{sinhg}.
\end{proof}

It follows from Lemma \ref{lem:gordon} and the fundamental theorem of surface theory that there exists a unique (up to ambient isometries) conformal immersion $\psi(u,v):\R^2\flecha \R^3$ with first and second fundamental forms given by 
\begin{equation}\label{fforms}
I= e^{2\omega} |d\zeta|^2, \hspace{1cm} II = e^{\omega}(\cosh \omega \, du^2 + \sinh \omega \, dv^2),
\end{equation}
where $\zeta=u+iv$. Indeed, $I,II$ satisfy the Gauss-Codazzi equations, by \eqref{sinhg}. We call this surface $\Sigma$. It has constant mean curvature $H=1/2$, and its Hopf differential is $Q=\esiz \psi_{\zeta \zeta},N\esde = 1/4$, where $N$ is the unit normal of $\psi$. The parametrization $\psi(u,v)$ is by curvature lines, and the principal curvatures $\kappa_1> \kappa_2$ of $\Sigma$ are 
\begin{equation}\label{princur}
\kappa_1 = e^{-\omega} \cosh \omega, \hspace{1cm} \kappa_2 =e^{-\omega} \sinh \omega.
\end{equation}
Note that $\kappa_1>0$, and that the sign of $\kappa_2$ coincides with the sign of $\omega$.  

\begin{definition}\label{def:Sigma}
We denote by $\Sigma$ the (unique) surface $\psi(u,v):\R^2\to\R^3$ with first and second fundamental forms given by \eqref{fforms} and initial conditions on its moving frame given by:
\begin{equation}\label{inimovi}
\psi(0,0)=(0,0,0), \hspace{0.3cm} N(0,0)=(1,0,0), \hspace{0.3cm} \frac{\psi_u(0,0)}{|\psi_u (0,0)|} = (0,0,1).
\end{equation}
\end{definition}




A fundamental property of $\Sigma$ for our purposes is given by the next result.

\begin{lemma}\label{lem:esfecur}
For each fixed $u\in \R$, the curvature line $v\mapsto \psi(u,v)$ of $\Sigma$ is \emph{spherical}, i.e., it is contained in a sphere or a plane $\mathbf{S}_u$ of $\R^3$.
\end{lemma}
\begin{proof}  

We first consider the case where $y\equiv 0$ (that, by \eqref{inicondi}, is equivalent to $\gamma=1$), and show that in this case all the curvature lines $v\mapsto \psi(u,v)$ are planar. Indeed, using \eqref{omu} and \eqref{fforms} it is straightforward to check that the vector $\eta:=e^{-\omega}\psi_u + z N$ does not depend on $v$. Thus, since $\esiz \psi_v,\eta\esde=0$, there exists $d=d(u)\in\R$ such that $\psi(u,v)$ is contained in the plane $\mathbf{S}_u = \{ q\in\R^3\,:\, \langle q,\eta(u)\rangle = d(u)\}$. 

Assume now that $y\not\equiv 0$. In particular, by \eqref{system1}, $y(u)$ can only have isolated zeroes.  Define now the function $\hat{c}=\psi - \frac{1}{y}e^{-\omega}\psi_u + \frac{y-z}{y}N$  on the open set $\{(u,v) \,:\, y(u)\neq 0\}$.  Let us check that $\hat{c}$ only depends on the variable $u$. 

Indeed, a direct computation using \eqref{fforms} shows that $\langle \hat{c}_v,\psi_u \rangle = \langle \hat{c}_v, N \rangle = 0$. Also, using \eqref{omu}, we have 
$$
\langle \hat{c}_v,\psi_v \rangle =  e^{\omega} \big( e^{\omega} - \frac{1}{y}\omega_u  - \frac{y-z}{y} \sinh\omega \big) = 0.$$
Thus,  $\hat{c} =\hat{c}(u)$ and we have  
\begin{equation}\label{gencen}
\hat{c}-\psi = -\frac{1}{y}e^{-\omega}\psi_u + \frac{y-z}{y} N.
 \end{equation}
This implies that for any $u\in\R$ with $y(u)\neq 0$, the curvature line $v\mapsto \psi(u,v)$ lies in a sphere ${\bf S}_u$ of center $\hat{c}(u)$ and radius $R=R(u)$ given by $R^2=(1+(z-y)^2)/y^2$.  At the values $u\in \R$ where $y(u)=0$, the radius become infinite and the curvature line $\psi(u,v)$ lies in a plane.   
\end{proof}

\begin{remark}\label{re:radan}
By Joachimsthal's theorem, $\Sigma$ intersects the sphere (or plane) $\mathbf{S}_u$ along $\psi(u,v)$ with a constant angle. It follows from the proof of Lemma \ref{lem:esfecur} that the values of $y,z$ are connected to the values of the radius $R$ and the intersection angle $\theta$ by 
\begin{equation}\label{radan}
R^2 = \frac{1+(z-y)^2}{y^2}, \hspace{1cm} \tan \theta=\frac{-1}{y-z}.
\end{equation}
Note that $R=\8$ if and only if $y=0$, and $\cos \theta=0$ if and only if $y=z$.

As a consequence,  and taking into account the initial conditions \eqref{inimovi}, the curvature line $\psi(0,v)$ lies the $x_3=0$ plane, and $\Sigma$ intersects this plane orthogonally along $\psi(0,v)$. In particular, $\Sigma$ is symmetric with respect to $x_3=0$.
\end{remark}

\begin{lemma}[\cite{W}]\label{lem:wen}
The centers $\hat{c}(u)$ of the spheres $\mathbf{S}_u$ lie in a common line $L$ of $\R^3$ parallel to $\psi_u(0,0)$; that is, in our case, $L$ is vertical. Moreover,   
\begin{equation}\label{eq:wen}
\hat c'(u)= \Big(0,0,  \frac{y'(0)}{y(u)^2}  \Big).
\end{equation}
\end{lemma}

\subsection{Free boundary nodoids}\label{sec:nodoid}


We finish this section by discussing the case in which $\alpha=1$ and $\gamma>1$ in the above construction.  In this situation, the polinomial $p(x)$ in \eqref{def:pab} has a positive double root at $\varrho_0=1/\gamma<1$, and \eqref{omini} implies that $\omega(0,v)\equiv \log\varrho_0$ for all $v\in \R$. By \eqref{omu} and \eqref{inicondi}, we obtain that $\omega_u$ vanishes along $u=0$.  
By uniqueness to the Cauchy problem for \eqref{sinhg}, we have $\omega(u,v)=\hat{\omega}(u)$, where $\hat{\omega}(u)$ is the solution to 
\begin{equation}\label{hatode}
\left\{\def\arraystretch{1.5}\begin{array}{l}\hat\omega'' + \sinh \hat\omega \cosh \hat\omega=0, \\ \hat\omega(0)=\log\varrho_0, \ \hat\omega'(0)=0.
\end{array}\right.
\end{equation}
Thus,  the first and the second fundamental forms of $\Sigma$ depend only on $u$,  and the surface is invariant under a $1$-parameter group of rigid motions of $\R^3$. Since the curves $v\mapsto \psi(u,v)$ are spherical, the surface must be rotationally invariant, and every such curve infinitely covers a circle in a horizontal plane. That is, $\Sigma_{\gamma}:=\psi(\R^2)\subset \R^3$ is a Delaunay surface with vertical axis and profile curve given by $u\mapsto \psi(u,0)$ (see the final comment in Remark \ref{re:radan}). 
The principal curvature associated to the profile curve is $\kappa_1 = \kappa_1(u) = e^{-\omega} \cosh \omega >0$,  which implies  that $\Sigma_{\gamma}$ must be a nodoid. Moreover, the nodoid is determined by the value of $\gamma>1$. Indeed, the planar geodesic $\psi(0,v)$, which has negative curvature $\kappa_2(0)<0$, describes the \emph{neck} of the nodoid and is a circle in the $x_3=0$ plane of radius 
\begin{equation}\label{necksize}
r_\gamma=\frac{1}{|\kappa_2 (0)|} = \frac{2}{\gamma^2-1}.
\end{equation}
This quantity gives the \emph{necksize} of the nodoid $\Sigma_{\gamma}$.

\begin{figure}
\begin{center}
\includegraphics[height=6cm]{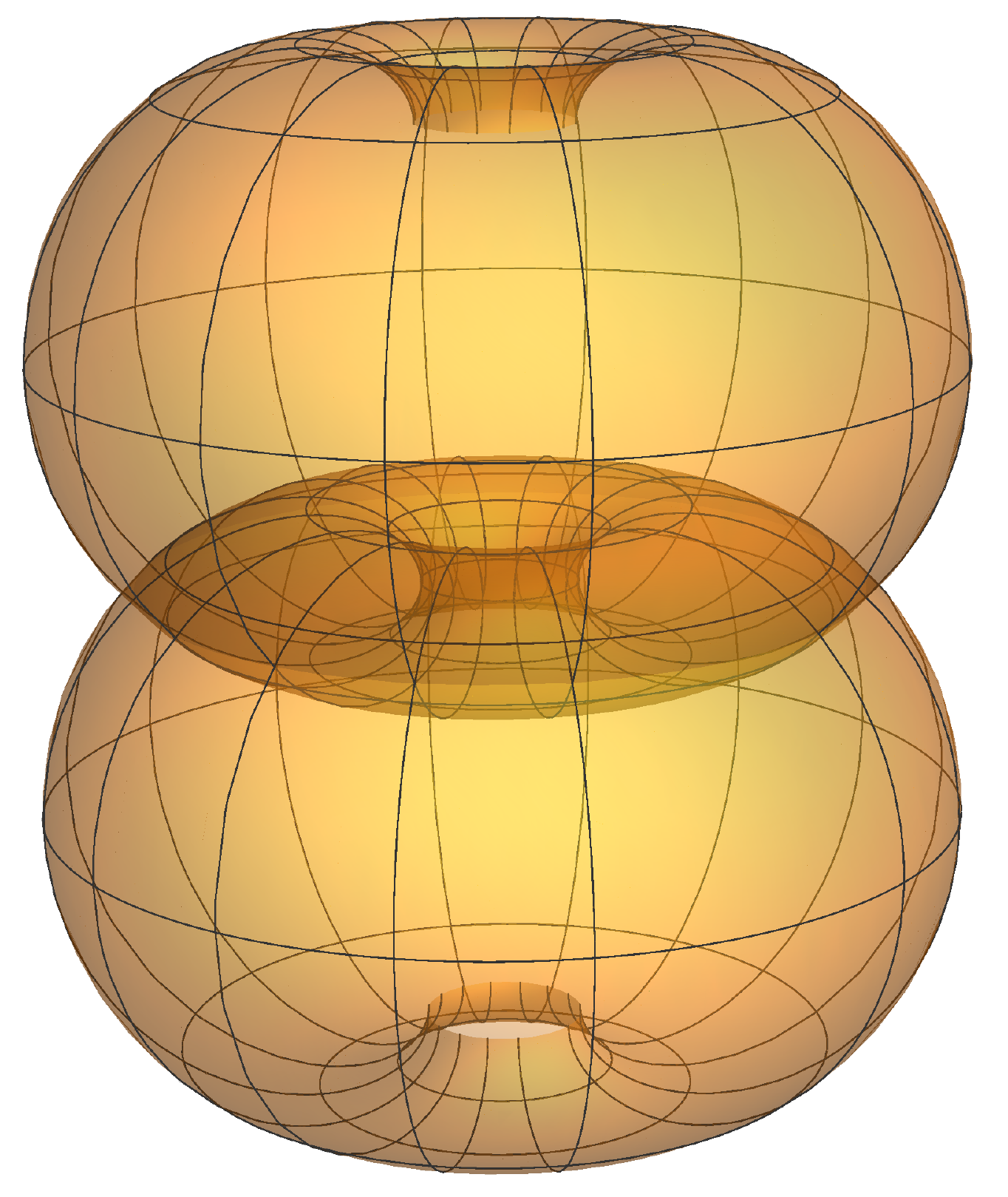}  \hspace{1cm}
\includegraphics[height=6cm]{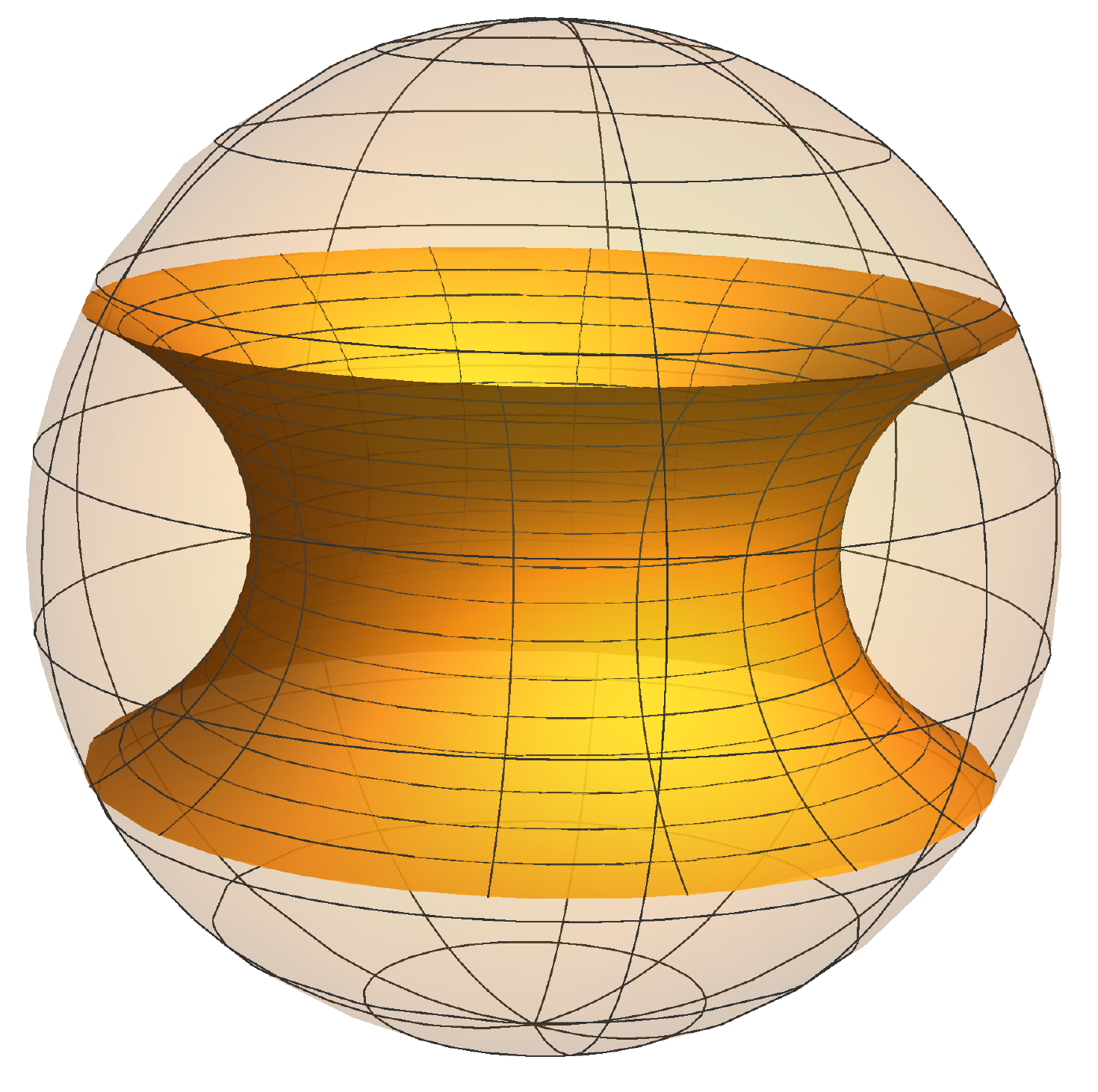} 
\caption{A nodoid $\Sigma_{\gamma}$ and its embedded free boundary compact piece ${\bf N}_{\gamma}$ in a ball.}\label{fig:eje1}
\end{center}
\end{figure}

It is clear that, for any $\gamma>1$, the nodoid $\Sigma_{\gamma}$ described above has a compact embedded piece ${\bf N}_{\gamma}$ that lies inside a ball of $\R^3$, and intersects the boundary of this ball orthogonally along $\parc {\bf N}_{\tau}$. See Figure \ref{fig:eje1}. More specifically, with respect to the parametrization $\psi(u,v)$ of the universal cover of the (complete) nodoid $\Sigma_{\gamma}$ above, let $\bar{u}>0$ be the first value where the tangent vector of its profile curve $u\mapsto \psi(u,0)$ is vertical. That is, $\psi(\bar{u},v)$ parametrizes a \emph{waist} of the nodoid $\Sigma_{\gamma}$. Then, since $\psi(u,0)$ has positive curvature and $\psi(0,v)$ parametrizes the \emph{neck} of $\Sigma_{\gamma}$, it is clear that there exists a unique value $\delta\in (0,\bar{u})$ such that the tangent line to $\psi(u,0)$ at $u=\delta$ passes through the intersection point $p_{\gamma}$ of the rotation axis of $\Sigma_{\gamma}$ and the $x_3=0$ plane where the neck of $\Sigma_{\gamma}$ is contained. In this situation, the restriction of the image of $\psi(u,v)$ to  $[-\delta,\delta]\times \R$ defines a compact, embedded rotational $H=1/2$ annulus ${\bf N}_{\gamma}$ contained in the ball $B_{\gamma}$ of $\R^3$ with center $p_{\gamma}$ and radius $|\psi(\delta,0)|$, and that intersects $\parc B_{\gamma}$ orthogonally along its boundary.

\begin{definition}\label{def:fbnodoid}
For any $\gamma>1$, we call ${\bf N}_{\gamma}$ the \emph{embedded free boundary nodoid} with $H=1/2$ and necksize $r_{\gamma}$, where $r_{\gamma}$ is given by \eqref{necksize}.\end{definition}

We remark that the unit normal of ${\bf N}_{\gamma}$ points \emph{outwards}, and that, by \eqref{necksize}, for any $\nu>0$ there is a unique embedded free boundary nodoid ${\bf N}_{\gamma}$ with necksize $r_{\gamma}=\nu$. Also, after a homothety, each ${\bf N}_{\gamma}$ determines a rotational embedded free boundary annulus in the unit ball $\B^3$, with constant mean curvature $H_{\gamma}=|\psi(\delta,0)|/2>0$.

For a more detailed description of rotational free boundary CMC annuli in $\B^3$, see \cite{BPS}.

\section{Symmetries and the period problem}\label{sec:3}

In this section we fix $(\alfa,\beta,\gamma)\in \cO$, with $\alfa>1$, and denote by $\Sigma=\Sigma(\alfa,\beta,\gamma)$ the CMC surface constructed in Section \ref{sec:2}. We will show that $\Sigma$ has a number of vertical planes of symmetry, and study when its spherical curvature lines are closed. To start, we show:

\begin{proposition}\label{pro:sim}
There exists $\sigma>0$ such that, for each $k\in \Z$, the curvature line $u\mapsto \psi(u,k\sigma)$ of $\Sigma$ lies in a vertical plane $\Omega_k$ of $\R^3$, and $\Sigma$ intersects $\Omega_k$ orthogonally along $\psi(u,k\sigma)$. This number $\sigma$ can be computed as
\begin{equation}\label{ecusigma}
\sigma = \int_{\varrho_0}^{\varrho_1} \frac{2}{\sqrt{p(x)}}dx,
\end{equation}
where $\varrho_0:=1/(\alfa \gamma)<\varrho_1:=\alfa/\gamma $ are the two positive roots of $p(x)$.

Moreover, the vectors $\eta_k:= \psi_v(0,k\sigma)$ are orthogonal to $\Omega_k$, and if we orient the planes $\Omega_k$ by $\eta_k$, the angle between $\Omega_k$ and $\Omega_{k+1}$ is independent of $k$, and given by 
\begin{equation}\label{ecutheta2}
\Theta= \int_{\varrho_0}^{\varrho_1} \frac{ x-1/x}{\sqrt{p(x)}} dx.
\end{equation}

The value $\Theta$ in \eqref{ecutheta2} lies in $(-\pi,\pi)$ and coincides with the total variation of the unit tangent vector of the planar geodesic $\psi(0,v)$ of $\Sigma$ between $v=k\sigma$ and $v=(k+1)\sigma$.
%
%
\end{proposition}
\begin{proof}
To start, following the notations of Section \ref{sec:2}, denote $x(v):=e^{\omega(0,v)},$ and recall that $x(0)=\varrho_0$. Then, by \eqref{omini}, there exists a value $\sigma>0$ such that $x$ is strictly increasing from $[0,\sigma]$ to $[\varrho_0,\varrho_1]$. In particular, $\sigma$ is the smallest positive number at which $x(\sigma)=\varrho_1$, and $x:\R\flecha [\varrho_0,\varrho_1]$ is $2\sigma$-periodic. Also by \eqref{omini}, we obtain the symmetry condition $\omega(0,k\sigma-v)= \omega(0,k\sigma +v)$, for any $k\in \Z$. Note that $\omega_v(0,v_0)=0$ if and only if $v_0=k\sigma$ for some $k\in\Z$.

By differentiating \eqref{omu}, we have  $$\omega_{uv}= (y(u) \sinh \omega + z(u)\cosh \omega) \omega_v,$$ which can be regarded as an ODE for $\omega_v(\cdot, v_0)$, for any $v_0\in \R$. Therefore, $\omega_v(u,k\sigma)\equiv 0$ for every $k\in \Z$, since $\omega_v(0,k\sigma)=0$. By uniqueness of the Cauchy problem for \eqref{sinhg} along $v=k\sigma$, we deduce that $$\omega(u,k\sigma +v)= \omega(u,k\sigma-v).$$ This implies by \eqref{fforms} and the fundamental theorem of surface theory that there exist planes $\Omega_k$, $k\in \Z$, in $\R^3$ such that:
\begin{enumerate}
\item
The curvature line $u\mapsto \psi(u,k\sigma)$ is contained in $\Omega_k$.
\item
$\Omega_k$ intersects $\Sigma$ orthogonally along $\psi(u,k\sigma)$, and it holds 
\begin{equation}\label{simka}
\psi(u,k\sigma-v)=\Psi_k \circ \psi (u,k\sigma+v)
\end{equation}
where $\Psi_k$ is the symmetry in $\R^3$ with respect to the plane $\Omega_k$.
\end{enumerate}
In particular, the vectors $\eta_k:= \psi_v(0,k\sigma)$ are orthogonal to $\Omega_k$. Therefore, since $\psi(0,v)$ lies in the horizontal $x_3=0$ plane, we see that all the $\Omega_k$ are vertical planes in $\R^3$.

Consider now the horizontal geodesic $\Gamma(v):=\psi(0,v)$ of $\Sigma$, and let $\xi(v):=e^{-\omega(0,v)} \Gamma'(v)$ denote its unit tangent vector. If $\theta(v)$ denotes the angle between $\xi(v)$ and a fixed direction of $\R^2$, the variation of $\theta(v)$ between $v=k\sigma$ and $v=(k+1)\sigma$ is given by $$\Theta = \int_{k\sigma}^{(k+1)\sigma} \esiz \xi'(v), J \xi(v)\esde dv,$$ where $J$ denotes the $\pi/2$-rotation in the plane $x_3=0$. From \eqref{inimovi} one has $J\xi(v)= N(0,v)$, and so, since $\kappa_2=e^{-\omega}\sinh\omega$, we obtain
\begin{equation}\label{ecutheta}
\Theta=\int_{k\sigma}^{(k+1)\sigma} \sinh \omega(0,v) dv.
\end{equation} The map $x(v)= e^{\omega(0,v)}$ satisfies \eqref{omini}, and defines a diffeomorphism between $[k\sigma, (k+1)\sigma]$ and $[\varrho_0,\varrho_1]$ that is strictly increasing (resp. decreasing) if $k$ is even (resp. odd). Thus, by \eqref{omini} and \eqref{ecutheta} we obtain after the change of variables $x=x(v)$ in \eqref{ecutheta} that, for any $k\in \Z$, $\Theta$ is given by \eqref{ecutheta2}, as desired. This change of variables also gives \eqref{ecusigma}, since
$$\sigma = \int_0^\sigma dv = \int_{\varrho_0}^{\varrho_1} \frac{2}{\sqrt{p(x)}}dx.$$

Note that the (oriented) angle between the planes $\Omega_k$ and $\Omega_{k+1}$, when oriented by their unit normals $\eta_k$, coincides with the angle between $\xi(k\sigma)$ and $\xi((k+1)\sigma)$. So, if we prove that $\Theta \in (-\pi,\pi)$ in \eqref{ecutheta2}, we obtain as a consequence that $\Theta$ describes the angle between $\Omega_k$ and $\Omega_{k+1}$. We do this next.

To start, use \eqref{def:pab} and \eqref{ecutheta2} to write 
\begin{equation}\label{thetaprod}
\Theta= \int_{\frac{1}{\alfa \gamma}}^{\frac{\alfa}{\gamma}} \frac{x-1/x}{\sqrt{(x+ \beta \gamma)(x+\frac{\gamma}{\beta})}} \, \frac{1}{\sqrt{-(x-\frac{\alfa}{\gamma})(x-\frac{1}{\alfa\gamma})} }dx.
\end{equation} 
The integral of the second factor of \eqref{thetaprod} on the interval $I:=[1/(\alfa \gamma),\alfa/\gamma]$ is equal to $\pi$. We will show next the first factor is bounded between $-1$ and $1$ when restricted to $I$, what proves that $\Theta$ takes values in $(-\pi,\pi)$. First of all, it is easy to check that this first factor has positive derivative. Thus, its maximum value, attained at $x=\alfa/\gamma$, is $$\frac{\alfa^2 -\gamma^2}{\alfa \sqrt{\alfa \gamma^2 (\beta+\frac{1}{\beta}) + \alfa^2 + \gamma^4}} < \frac{{\rm max} \{\alfa^2 -\gamma^2,0\}}{\alfa \sqrt{2\alfa \gamma^2 + \alfa^2 + \gamma^4}} = \frac{{\rm max} \{\alfa^2 -\gamma^2,0\}}{\alfa(\alfa + \gamma^2)} <1.$$ By a similar argument, the minimum value, attained at $x= 1/(\alfa \gamma)$, is greater than $-1$. This completes the proof.
\end{proof}

\begin{remark}\label{rem:gamma1}
We will prove in Section \ref{sec:5} that when $\gamma=1$ we have $\Theta=0$, and when $\gamma>1$, it holds $\Theta<0$.
\end{remark}

\begin{corollary}\label{cor:sim}
Assume that $\Theta/\pi = m /n \in \Q\cap (-1,1)$, with $\Theta\neq 0$. Then, $\psi(u,v+ 2n \sigma)=\psi(u,v)$. In particular, by considering the quotient of $\R^2$ by the relation $(u,v)\sim (u,v+2n\sigma)$, we have:
\begin{enumerate}
\item
$\Sigma$ has the topology of an annulus.
 \item
$\Sigma$ is symmetric with respect to the $x_3=0$ plane, and with respect to $n$ equiangular vertical planes $\Omega_1,\dots, \Omega_n$ intersecting along a common vertical line. This line is actually the vertical line $L$ of Lemma \ref{lem:wen}.
\item
The planar geodesic $\psi(0,v): [0,2n\sigma]\equiv \S^1\flecha \{x_3=0\}$ has rotation index $m$.
\end{enumerate}
\end{corollary}
\begin{proof}
By Proposition \ref{pro:sim}, $\Sigma$ is symmetric with respect to, at least, $n$ \emph{different} vertical equiangular planes $\Omega_1,\dots, \Omega_n$. Note that $n\geq 2$, since $\Theta/\pi \in (-1,1)$. By Proposition \ref{pro:sim}, for any $u\in \R$ and any $k\in \{1,\dots, n\}$, the curve $\psi(u,k\sigma)$ lies in the symmetry plane $\Omega_k$. Since $\Theta \neq 0$, it follows from Remark \ref{rem:gamma1} that $\gamma > 1$, which implies by \eqref{inicondi} that $y(u)$ is not identically zero. As a consequence, there must exist $u_0$ for which \eqref{gencen} holds. From this equation we deduce that $\hat{c}(u_0)$ lies in $\Omega_k$. The planes $\Omega_k$ are vertical, so the vertical line $L$ that contains the centers $\hat{c}(u)$ must also lie in $\Omega_k$. Therefore, $L \subset \cap_{k=1}^n \Omega_k$. Since $n\geq 2$ and the planes are different, we deduce that $L= \cap_{k=1}^n \Omega_k$, as claimed.



 
In particular, we obtain directly from \eqref{simka} that $\psi(u,v+2\sigma)=\cR (\psi(u,v))$ where $\cR$ is the rotation of angle $2\Theta$ around the intersection line $L$. Now, since $2\Theta = 2\pi (m/n)$, we have from this equation that $\psi(u,v)$ is $2n\sigma$-periodic in the $v$-direction, as stated. In particular, by considering the quotient of $\R^2$ by the relation $(u,v)\sim (u,v+2n\sigma)$, we deduce that $\Sigma$ is, topologically, an immersed annulus in $\R^3$. The fact that $\Sigma$ is symmetric with respect to $x_3=0$ was proved in Section \ref{sec:2}, see Remark \ref{re:radan}. Assertion (3) follows directly from Proposition \ref{pro:sim}. This completes the proof of Corollary \ref{cor:sim}.
\end{proof}

\section{The free boundary condition}\label{sec:4}
In this section we will show that, when $(\alfa,\beta,\gamma)$ varies in a certain region $\cW$ of the parameter domain $\cO\subset \R^3$, we can control the orthogonal intersection of the spherical curvature lines of $\Sigma$ with the spheres that contain them, by studying system \eqref{system1}.


So, to start, let us consider $(\alfa,\beta,\gamma)\in \cO$, and let $(y(u),z(u)):\R\flecha \R^2$ be the solution to system \eqref{system1} with the initial conditions \eqref{inicondi}. Thus, $y(0)=z(0)=0$. We first prove:
%

\begin{proposition}\label{prosistem1}
If $\gamma>1$, then there exists a unique $u_1>0$ such that:
\begin{enumerate}
\item
$y(u_1)=0$.
\item
$y(u)>0$ for every $u\in (0,u_1)$.
 \item
If $z(u)$ is not identically zero, then $z(u)\neq 0$ for every $u\in (0,u_1]$.
\end{enumerate}
\end{proposition}
\begin{proof}
Wente noted in \cite[p.17]{W} that the system \eqref{system1} (for an arbitrary constant $\hat{a}$) has a Hamiltonian nature, and admits two first integrals, namely,
\begin{equation}\label{def:h}
y'^2-z'^2 - (\hat{a}-1)y^2 + \hat{a} z^2 + (y^2-z^2)^2 = h \in \R
\end{equation}
and
\begin{equation}\label{def:k}
(z y'-y z')^2 + z'^2 + z^2 (y^2-z^2 -\hat{a}) = k\in \R.
\end{equation}

These first integrals allow us to separate variables in system \eqref{system1}, following a classical procedure by Jacobi. Specifically, as detailed in \cite[p.17]{W}, one starts by considering the change of variables 
\begin{equation}\label{change}
y^2 = -(1-s) (1-t), \hspace{1cm} z^2 = -s t.
\end{equation}
This defines a diffeomorphism from any of the four open quadrants of the $(y,z)$-plane onto the region of the $(s,t)$-plane given by $\{s> 1, t< 0\}$, which extends homeomorphically to the boundary. After this change, any solution to system \eqref{system1} transforms into a solution $(s(\landa),t(\landa))$ to system
\begin{equation}\label{system2}
\left\{\def\arraystretch{1.3} \begin{array}{lll} s'(\landa)^2 & = & s (s-1) q(s), \hspace{0.5cm} (s\geq 1), \\ t'(\landa)^2 & = & t (t-1) q(t), \hspace{0.6cm} (t\leq 0),\end{array} \right.
\end{equation}
where $q(x)$ is the third-degree polynomial
\begin{equation}\label{def:q}
q(x)=-x^3 + (\hat{a}+1) x^2 + (h-\hat{a})x + k,
\end{equation}
and the new parameter $\landa$ is related to $u$ by 
\begin{equation}\label{changecor}
2 u'(\landa)= s(\landa)-t(\landa)>1.
\end{equation}

Let now $(y(u),z(u))$ be the solution to \eqref{system1} that we started with, i.e., the one with the initial conditions \eqref{inicondi}, and the choice $\hat a:= 1-AB+C^2$. Then, from \eqref{inicondi}, one can write $h,k$ in \eqref{def:h}, \eqref{def:k} in terms of $(A,B,C)$, and deduce from there that $q(x)$ in \eqref{def:q} can be rewritten as
\begin{equation}\label{pola3}
q(x)=-\left(x-(C^2+1)\right)\left(x^2-(1-AB)x+\frac{(A-B)^2}{4}\right) =: -(x-r_3)h(x).
\end{equation}
Here, we recall that $A,B,C$ are given  in terms of $(\alfa,\beta,\gamma)$ by \eqref{const2}.

Note that $q(x)$ always has a simple positive root $r_3:=C^2+1> 1$, and that $q(0)=0$ if and only if $\alfa=\beta$. If $\alfa\neq \beta$, then by \eqref{pola3}, $q(x)$ has two negative, possibly coinciding roots $r_1\leq r_2<0$, besides the positive one at $r_3$. 

Let $(s(\landa),t(\landa))$ be the solution to \eqref{system2} obtained from $(y(u),z(u))$. Since $y(0)=z(0)=0$, we have from \eqref{change} that $s(0)=1$, $t(0)=0$, possibly after a translation in the $\landa$-parameter. Let us observe that each equation of \eqref{system2} is a first order autonomous ODE, that can be implicitly solved by integration. From here, we see that both $s(\landa)$ and $t(\landa)$ are defined for every $\landa \in \R$. Moreover, $s$ is periodic with $s(\landa+ 2\cL)=s(\landa)$, where
\begin{equation}\label{stperiod1}
\cL:=\int_1^{r_3} \frac{dx}{\sqrt{x (x-1) q(x)}}<\8.
\end{equation}
Similarly, if $r_2$ is a simple root of $q(x)$, then $t$ is periodic with $t(\landa + 2\cM)=t(\landa),$ where the corresponding half-period is
\begin{equation}\label{stperiod2}
\cM:=\int_{r_2}^{0} \frac{dx}{\sqrt{x (x-1) q(x)}}<\8.
\end{equation}
If $r_2$ is a double root of $q(x)$, then $\cM=\8$ and $t$ is not periodic.

In this way, $$s(2\cL)=s(0)=1, \hspace{0.5cm} s(\cL)=r_3,$$ and either $s(\landa)\equiv 1$, or $s(\landa)\in (1,r_3]$ for every $\landa\in (0,2\cL)$. Note that $s(\landa)\equiv 1$ is actually not possible, by the first equation in \eqref{change} and the condition $y'(0)>0$.

In a similar fashion, if $\cM<\8$, then $$t(2\cM)=t(0)=0, \hspace{0.5cm} t(\cM)=r_2,$$ and either $t(\landa)\equiv 0$ or $t(\landa)\in [r_2,0)$ for every $\landa\in (0,2\cM)$. If $\cM=\8$, then $t(\landa)\to r_2$ as $\landa\to \8$. Moreover, in both cases, $t(\landa)\equiv 0$ happens if and only if $z'(0)=0$ in \eqref{inicondi}, i.e., if and only if $\alfa=\beta$.

The parameters $u$ and $\landa$ are related by \eqref{changecor}. From there, $u'(\landa)$ is bounded between two positive constants, and so there exists a unique value $u_1:=u(2\cL)>0$ such that the restriction of $u(\landa)$ to $[0,2\cL]$ is a diffeomorphism onto $[0,u_1]$. By the first equation in \eqref{change} and the above properties of $s(\landa)$, we deduce that items (1) and (2) of the Proposition hold.

In order to prove item (3), in view of the second equation in \eqref{change}, it suffices to show that $\cL<\cM$ when $t(\landa)$ is not identically zero. We do this next.

To start, consider the function $g(x)=1/(x-w_0)$, where $w_0\in (0,1)$ is a number that will be determined later on. We can then use the change of variable $\mu=g(x)$ in the intervals $[1,r_3]$ and $[r_2,0]$, together with the general relation for $g$ $$x-a= - \frac{g(x) -g(a)}{g(x)g(a)} = -\frac{ \mu-g(a)}{\mu g(a)},$$ to obtain from \eqref{stperiod1}, \eqref{stperiod2} and \eqref{pola3} the alternative expressions for $\cL$ and $\cM$
\begin{equation}\label{altLM}
\cL = \int_{g(r_3)}^{g(1)} \Psi(\mu) d\mu, \hspace{0.5cm} \cM = \int_{g(0)}^{g(r_2)} \Psi(\mu) d\mu,
\end{equation}
where $$\Psi(\mu)= \frac{\sqrt{g(r_2) g(0)g(1)g(r_3)} }{\sqrt{-(\mu-g(r_2))(\mu-g(0))(\mu-g(1))(\mu-g(r_3))}\sqrt{\frac{1}{\mu} +w_0-r_1}}.$$ Since $\mu<0$ for every $\mu\in (g(0),g(r_2))$, we have from \eqref{altLM} that $\cM>T$, where $$T:=\int_{g(0)}^{g(r_2)}  \frac{\sqrt{g(r_2) g(0)g(1)g(r_3)} \, d\mu }{\sqrt{-(\mu-g(r_2))(\mu-g(0))(\mu-g(1))(\mu-g(r_3))}\sqrt{w_0-r_1}}.$$ Similarly, since $\mu>0$ when $\mu\in (g(r_3),g(1))$, we obtain $\cL<T'$, where
\begin{equation}\label{tprima}
T':=\int_{g(r_3)}^{g(1)}  \frac{\sqrt{g(r_2) g(0)g(1)g(r_3)} \, d\mu }{\sqrt{-(\mu-g(r_2))(\mu-g(0))(\mu-g(1))(\mu-g(r_3))}\sqrt{w_0-r_1}}.
\end{equation}

We next show that $T=T'$ for a certain choice of $w_0\in (0,1)$. Consider the function $f(x,w):=1/(x-w)$, where $w\in (0,1)$, and define $$F(w):= f(1,w)-f(r_3,w)-(f(r_2,w)-f(0,w)): (0,1)\flecha \R.$$ Since $F(w)\to -\8$ (resp. $F(w)\to \8$) as $w\to 0$ (resp. $w\to 1$), there exists some $w_0\in (0,1)$ for which $F(w_0)=0$, by continuity of $F$. If we define $w_0$ in this way, we obtain that $$g(1)-g(r_3)= g(r_2)-g(0)>0.$$ From this expression, we obtain by applying the change of variable $\xi=-\mu+g(1)+g(0)$ in \eqref{tprima} that

$$T' =  - \int_{g(r_2)}^{g(0)}  \frac{\sqrt{g(r_2) g(0)g(1)g(r_3)} \, d\xi }{\sqrt{-(\xi-g(r_2))(\xi-g(0))(\xi-g(1))(\xi-g(r_3))}\sqrt{w_0-r_1}}  =  T.$$ Thus, $\cL< T' =T<\cM$, what yields our claim and completes the proof of Proposition \ref{prosistem1}.
\end{proof}

\begin{remark}\label{rem:u1an}
When $u_1$ is viewed as a function $u_1=u_1(\alfa,\beta,\gamma): \cO\cap \{\gamma>1\}\flecha \R$, it is real analytic. This follows from the standard analyticity theorem of solutions to ODEs with respect to initial conditions and parameters and \eqref{stperiod1}, \eqref{changecor}, since $u_1=u(2\cL)$.
\end{remark}

The next result will allow us to control the free boundary condition. Let us recall that the spherical curvature line $\psi(u_0,v)$ intersects orthogonally the sphere (or plane) ${\bf S}_{u_0}$ where it lies if and only if $y(u_0)=z(u_0)$; see Remark \ref{re:radan}.

Given $(\alfa,\beta,\gamma)\in \cO$ and $A,B,C$ as in \eqref{const2}, we denote 
\begin{equation}\label{defi:W}
\cW :=\left\{(\alfa,\beta,\gamma)\in \cO : \beta\geq \alfa,  \, C^2>\frac{(A-B)^2}{4AB}\right\}.
\end{equation}
We remark that if $(\alfa,\beta,\gamma)\in \cW$, then $\alfa\geq 1$ and $\gamma>1$.

\begin{proposition}\label{prosistem2}
Let $(\alfa,\beta,\gamma)\in \cW$. 
%
Then, there is a unique $\tau\in (0,u_1]$ such that $y(\tau)=z(\tau)$, and $y(u)>z(u)$ for every $u\in (0,\tau)$. 

Moreover, the map $(\alfa,\beta,\gamma)\in \cW \mapsto \tau(\alfa,\beta,\gamma)$ is real analytic.
%
\end{proposition}
\begin{proof}
Define 
\begin{equation}\label{deffi}
\phi(u):=y(u)-z(u).
\end{equation}
By \eqref{inicondi} and
\begin{equation}\label{fbcon}
C^2>\frac{(A-B)^2}{4AB},
\end{equation}
we have $\phi(0)=0$ and $\phi'(0)>0$. On the other hand, by Proposition \ref{prosistem1} we have $\phi(u_1)\leq 0$, with equality holding only if $z(u)\equiv 0$. Thus, there exists a unique value $\tau\in (0,u_1]$ that satisfies the stated conditions.

We now prove that the map $(\alfa,\beta,\gamma)\in \cW \mapsto \tau(\alfa,\beta,\gamma)$ is real analytic. Consider the real analytic mapping $$\Phi(u,\alfa,\beta,\gamma) = y(u)-z(u):\R\times \cW\flecha \R,$$ where we view $y(u),z(u)$ as functions depending also on $(\alfa,\beta,\gamma)$. We claim that $\Phi_u\neq 0$ at any point $w\in \R\times \cW$ of the form 
\begin{equation}\label{wtau}
w:=(\tau(\alfa,\beta,\gamma),\alfa,\beta,\gamma).
\end{equation} 
This would imply by the implicit function theorem that $\tau(\alfa,\beta,\gamma)$ is real analytic on $\cW$.


To prove this claim, arguing by contradiction, assume that $\Phi_u(w)=0$ for $w$ as in \eqref{wtau}. Then, we must have $y(\tau)=z(\tau)$ and $y'(\tau)=z'(\tau)$. Thus, it follows from \eqref{system1} that $y''(\tau)-z''(\tau)=-y(\tau)\leq 0$. We distinguish two cases.

If $\tau\neq u_1$, then $y(\tau)\neq 0$ by item (2) of Proposition \ref{prosistem1}, and so $\phi''(\tau)<0$. Thus, $y(u)-z(u)$ has a maximum at $u=\tau$, in contradiction with the fact that $y(u)-z(u)>0$ for every $u\in (0,\tau)$. 

If $\tau= u_1$, then $z(u)\equiv 0$, by item (3) of Proposition \ref{prosistem1}. So, since $y(\tau)=0$ in this case, we deduce using the first integral \eqref{def:h} that $y'(\tau)^2 = y'(0)^2$. This is a contradiction, since $y'(\tau)=z'(\tau)=0$ but $y'(0)>0$.

This proves the desired analyticity. Note that the above argument shows that $\Phi_u(w)<0$.
\end{proof}

\begin{remark}\label{rem:ymenorz}
Assume again that $\gamma>1$, that $\beta\geq \alfa\geq 1$, but now
\begin{equation}\label{fbcon2}
C^2\leq \frac{(A-B)^2}{4AB}.
\end{equation}
Then, $y(u)<z(u)$ for every $u>0$ small enough. 

Indeed, if we have the strict inequality in \eqref{fbcon2}, the same argument as in the proof of Proposition \ref{prosistem2} shows that $\phi(0)=0$ and $\phi'(0)<0$, where $\phi$ is given by \eqref{deffi}; this proves the claim. Now, if equality holds in \eqref{fbcon2}, then a computation using \eqref{system1} shows that $\phi'(0)=\phi''(0)=0$ and $\phi'''(0)=-y'(0)<0$. Again, we have $\phi(u)<0$ for $u>0$ sufficiently small, and the claim also holds in this case.
\end{remark}


\section{Proof of Theorem \ref{th:main}}\label{sec:5}

\subsection{The period map}

Given $(\alpha,\beta,\gamma) \in \cO$ such that $\alpha > 1$, we define the \emph{period map} ${\rm Per}(\alpha,\beta,\gamma)$ as
\begin{equation}\label{eq:Per}
{\rm Per}(\alpha,\beta,\gamma) = \frac{\Theta}{\pi}= \frac{1}{\pi}\int_{\varrho_0}^{\varrho_1} \frac{ x-1/x}{\sqrt{p(x)}} dx,
\end{equation}
where $\varrho_0<\varrho_1$ are the two positive roots of $p(x)$; see \eqref{def:pab}, \eqref{ecutheta2}. 

\begin{proposition}\label{perana}
The period map ${\rm Per}(\alpha,\beta,\gamma)$ can be extended analytically to the set $\{(\alpha,\beta,\gamma) \; :\; \alpha, \beta > 0, \gamma \geq 1\}$. Moreover, for this analytic extension we have \begin{equation}\label{eq:invper}
    {\rm Per}\left(\alpha,\beta,\gamma\right)={\rm Per}\left(\frac{1}{\alpha},\beta,\gamma\right) = {\rm Per}\left(\alpha,\frac{1}{\beta},\gamma\right).
\end{equation}
Additionally, if $\alpha = 1$,
\begin{equation}\label{eq:pera1}
\begin{aligned}
{\rm Per}(1,\beta,\gamma) = \frac{1 - \gamma^2}{\sqrt{1 + (\beta + \frac{1}{\beta})\gamma^2 + \gamma^4}}.
\end{aligned}
\end{equation}
\end{proposition}

\begin{proof}
We observe that the polynomial \eqref{def:pab} can be defined for all $\alpha,\beta > 0$. In fact, $p(x)$ is invariant under the changes $\alpha \to 1/\alfa$ and $\beta \to 1/\beta$. As a consequence, the definition of ${\rm Per}(\alpha,\beta,\gamma)$ in \eqref{eq:Per} can be extended to every $\alpha,\beta,\gamma > 0$ such that $\alpha \neq 1$. Equation \eqref{eq:invper} follows immediately from this definition.

Applying the change of variables 
\begin{equation}\label{chava}
x = h_\alpha(t) := \left(\frac{\alpha}{\gamma} - \frac{1}{\alpha \gamma}\right)t + \frac{1}{\alpha \gamma},
\end{equation} 
we can rewrite 
$${\rm Per}(\alpha,\beta,\gamma) = \frac{1}{\pi} \int_0^1 \frac{\left(h_\alpha(t) - \frac{1}{h_\alpha(t)}\right)}{\sqrt{t(1-t)}\sqrt{h_\alpha^2(t) + \left(\beta + \frac{1}{\beta}\right)\gamma h_\alpha(t) + \gamma^2}}\, dt.$$
\noindent This expression allows us to extend analytically the definition of ${\rm Per}(\alpha,\beta,\gamma)$ for $\alpha = 1$. Indeed, if $\alpha = 1$ then $h_\alpha(t) \equiv \frac{1}{\gamma}$, and so
$${\rm Per}(1,\beta,\gamma) = \frac{1}{\pi} \int_0^1 \frac{\left(\frac{1}{\gamma} - \gamma\right) }{\sqrt{t(1-t)}\sqrt{\frac{1}{\gamma^2} + \left(\beta + \frac{1}{\beta}\right)+ \gamma^2}} \, dt= \frac{1 - \gamma^2}{\sqrt{1 + (\beta + \frac{1}{\beta})\gamma^2 + \gamma^4}},$$ what gives \eqref{eq:pera1} and completes the proof.
\end{proof}

\begin{remark}\label{sigmana}
Similarly to the definition of the period map in \eqref{eq:Per}, we can define $\sigma = \sigma(\alpha,\beta,\gamma):\cO \cap \{\alpha > 1 \}\to \R$ as the map given by \eqref{ecusigma}. The change of variables \eqref{chava} can also be used to extend $\sigma(\alpha,\beta,\gamma)$ analytically to the set $\{(\alpha,\beta,\gamma) : \alpha, \beta > 0, \gamma \geq 1\}$. In fact, arguing as in Proposition \ref{perana}, this extension also satisfies
\begin{equation}
\sigma(\alpha,\beta,\gamma) = \sigma\left(\frac{1}{\alpha},\beta,\gamma\right) = \sigma\left(\alpha,\frac{1}{\beta},\gamma\right).
\end{equation}
Moreover, if $\alpha = 1$, then
\begin{equation}\label{eq:sigma1}
\sigma(1,\beta,\gamma) = \frac{2 \pi \gamma}{\sqrt{1 + (\beta + \frac{1}{\beta})\gamma^2 + \gamma^4}}.
\end{equation}
\end{remark}

\begin{proposition}\label{pro:period}
The following assertions hold on the set $\{(\alfa,\beta,\gamma): \alfa>0, \beta>0, \gamma \geq 1\}$:
\begin{enumerate}
       \item $\frac{\partial {\rm Per}}{\partial \gamma} < 0$.
    \item ${\rm Per}(\alpha,\beta,1) = 0$, and $\lim_{\gamma \to \infty} {\rm Per}(\alpha,\beta,\gamma) = -1$. 
    \item For every $c\in (-1,0]$, the level set ${\rm Per}^{-1}(c)$ is the graph $\gamma=\gamma_c(\alfa,\beta)$ of an analytic map $\gamma_c(\alpha,\beta)$ defined for all $\alpha,\beta > 0$. Moreover, $\gamma_0 \equiv 1$ and $\gamma_c < \gamma_c'$ for all $c > c'$.
\end{enumerate}
In particular, ${\rm Per}(\alfa,\beta,\gamma)\in (-1,0)$ if $\gamma >1$.
\end{proposition}

\begin{proof}
To prove item (1), we first observe that the function $g(x;a,b)$ defined as
\begin{equation}\label{def:ge}
g(x;a,b) = \frac{1 - x^2}{\sqrt{a +b x^2 + x^4}}
\end{equation}
satisfies $g_x< 0$ for all $x, a, b > 0$. Also, note that by the symmetry condition \eqref{eq:invper}, we can assume that $\alpha \geq 1$. If $\alpha = 1$, we have by \eqref{eq:pera1} that 
\begin{equation}\label{peralfa1}
{\rm Per}(1,\beta,\gamma) = g(\gamma;1, \beta + 1/\beta),
\end{equation} 
so clearly $\frac{\partial {\rm Per}}{\partial \gamma}(1,\beta,\gamma) < 0$. Similarly, if $\alpha > 1$, we consider the change of variable $\mu=1/(\gamma x)$ in \eqref{eq:Per} to obtain
\begin{equation}\label{eq:perg}
   {\rm Per}(\alpha,\beta,\gamma) = \frac{1}{\pi}\int_\frac{1}{\alpha}^\alpha \frac{g(\gamma \mu;\mu^2, \left(\beta +1/\beta\right)\mu)}{\sqrt{(\alpha - \mu)(\mu -1/\alfa)}}\, d\mu. 
\end{equation}  
Thus, $\frac{\partial {\rm Per}}{\partial \gamma}$ is the integral of a strictly negative function, and so $\frac{\partial {\rm Per}}{\partial \gamma}< 0$. This proves the first assertion.

We next show that ${\rm Per}(\alpha,\beta,1) \equiv 0$. The case $\alpha = 1$ is immediate by \eqref{eq:pera1}. For $\alpha > 1$, we observe that the polynomial $p(x)$ defined in \eqref{def:pab} satisfies the relation $p(x) = x^4\, p(1/x)$ when $\gamma = 1$. By considering the change of variable $\mu =1/x$, we obtain
$$
{\rm Per}(\alpha,\beta,1) = \frac{1}{\pi}\int_\frac{1}{\alpha}^\alpha \frac{ x-1/x}{\sqrt{p(x)}} dx = \frac{1}{\pi}\int_\frac{1}{\alpha}^\alpha \frac{ 1/\mu-\mu}{\sqrt{\mu^4 \, p(1/\mu)}} d\mu = -{\rm Per}(\alpha,\beta,1),$$
and therefore ${\rm Per}(\alpha,\beta,1) \equiv 0$.

To complete the proof of item (2), it remains to show that ${\rm Per}(\alfa,\beta,\gamma)\to -1$ as $\gamma\to \8$. The case $\alfa=1$ follows immediately from \eqref{peralfa1} and the fact that $\lim_{x \to \infty} g(x;a,b) = -1$, for any $a,b>0$. Similarly the case $\alpha > 1$ follows from \eqref{eq:perg} and the equality
$$\int_\frac{1}{\alpha}^\alpha \frac{d\mu}{\sqrt{(\alpha - \mu)(\mu - 1/\alfa)}} =\pi.$$ This proves item (2).


As a result of the first two assertions, we have that ${\rm Per}(\alfa,\beta,\gamma)\in (-1,0]$ when $\alfa,\beta>0$ and $\gamma\geq 1$. Moreover, for any $c\in (-1,0]$, the level set ${\rm Per}^{-1}(c)$ is a graph $\gamma= \gamma_c(\alfa,\beta)$, where $\gamma_c$ is a real analytic function on $(0,\8)\times (0,\8)$, with $\gamma_c\geq 1$. Note that the real analyticity of $\gamma_c$ follows from item (1) and the analyticity of the period map, proved in Proposition \ref{perana}. Also, note that $\gamma_0(\alpha,\beta) \equiv 1$, by item (2). By the monotonicity of ${\rm Per}$ with respect to $\gamma$, we have that $\gamma_c< \gamma_{c'}$ whenever $c> c'$.  This completes the proof.
\end{proof}

\begin{remark}
Let $\alpha = 1$. Defining $B = \frac{1}{2}(\beta +1/\beta)$, $C_c := \frac{1}{2}\left(\gamma_c(1,\beta) 
 -1/\gamma_c(1,\beta) \right)$, equation \eqref{eq:pera1} can be rewritten as
\begin{equation}\label{eq:Ckappa}
    C_c^2 = \frac{c^2(B + 1)}{2(1 - c^2)}.
    \end{equation}
\end{remark}

\subsection{Existence of CMC annuli with spherical free boundary}


In the next Lemma we use the notation $u_1=u_1(\alfa,\beta,\gamma)$ introduced in Proposition \ref{prosistem1} and Remark \ref{rem:u1an}.

\begin{lemma}\label{ustar} For any $(\alpha,\beta,\gamma)\in \cO\cap \{\gamma>1\}$ there exists a unique value $u^* \in (0,u_1)$ such that $\hat{c}_3(u^*)=0$, where $\hat{c}_3(u)$ denotes the third coordinate of the center of the sphere $\mathbf{S}_u$ (see Lemma \ref{lem:wen}). Moreover, the map $u^*(\alfa,\beta,\gamma): \cO\cap \{\gamma>1\} \flecha\R$ is real analytic.
\end{lemma}
\begin{proof} 
To start, fix $(\alpha,\beta,\gamma)\in \cO\cap \{\gamma>1\}$. Then, we have $y'(0)>0$, by \eqref{inicondi}. According to \eqref{eq:wen}, this means that $\hat{c}_3$ is strictly increasing. Moreover, for $u\in (0,u_1)$, we have from \eqref{radan} that $\hat{c}_3(u) \to -\infty$ as $u \to 0$ and $\hat{c}_3(u) \to \infty$ as $u \to u_1$. So, there is a unique $u^* \in (0,u_1)$ where $\hat{c}_3(u^*)=0$.



We will prove next that the map $(\alfa,\beta,\gamma)\in \cO\cap \{\gamma>1\} \mapsto u^*(\alpha,\beta,\gamma)$ is real analytic. To do so, we define first
$$\hat{C}_3(u,\alpha,\beta,\gamma) := \hat{c}_3(u),$$
\noindent that is, we see $\hat{c}_3$ as a function depending also on the parameters $(\alpha,\beta,\gamma)$. Here, we assume that $u\in (0,u_1)$, and so $\hat{C}_3(u,\alfa,\beta,\gamma)$ takes finite real values. Note that $\hat{C}_3(u,\alfa,\beta,\gamma)=0$ if and only if $u=u^*(\alfa,\beta,\gamma)$. Also, 
 $(\hat{C}_3)_u> 0$, by \eqref{eq:wen}. Hence, by the implicit function theorem, $u^*(\alfa,\beta,\gamma)$ will be real analytic if so is $\hat{C}_3$.
 
Now, according to \eqref{gencen}, $\hat{C}_3$ will be analytic if so are $y(u), z(u), \psi(u,0)$ when seen as functions depending on $(\alpha,\beta,\gamma)$. This analyticity for $y(u),z(u)$ is immediate, by \eqref{system1} and the usual regularity theory for ODE systems. The same can be said about $\psi(u,0)$: this curve is the solution to the differential system
\begin{equation}
\left\{\def\arraystretch{1.3} \begin{array}{lll} \omega_u & = & y(u) \cosh\omega + z(u) \sinh\omega,\\ 
\psi_{uu} & = & \omega_u\psi_u + \kappa_1e^{2\omega} N ,\\
N_u  & = & -\kappa_1\psi_u ,\\
\end{array} \right.
\end{equation}
where $\kappa_1 = e^{-\omega} \cosh \omega$, with the initial conditions
\begin{equation}
\omega(0,0)=\frac{1}{\alpha \gamma}, \hspace{0.2cm} \psi(0,0)=(0,0,0), \hspace{0.3cm} \psi_u(0,0)= (0,0,e^{\omega(0,0)}), \hspace{0.3cm} N(0,0)= (1,0,0),
\end{equation}
\noindent which depends analytically on $(\alpha,\beta,\gamma) \in \cO$.
\end{proof}
%
%
%
%

\begin{definition}\label{def:sigma0}
We will let $\Sigma_0$ denote the subset of the surface $\Sigma$ of Definition \ref{def:Sigma} given by the restriction of the immersion $\psi(u,v)$ to $[-u^*,u^*]\times \R$.
\end{definition}

\begin{proposition}\label{thm:Per1n}

Let $(\alpha,\beta,\gamma)\in \cO\cup \{\gamma>1\}$ so that ${\rm Per}(\alpha,\beta,\gamma)=-1/n$ for some $n\in \N$, $n\geq 2$. Then, $\psi(u,v+ 2n \sigma)=\psi(u,v)$ (see Remark \ref{sigmana}). In particular, by considering the quotient of $[-u^*,u^*]\times \R$ by the relation $(u,v)\sim (u,v+2n\sigma)$, we can view $\Sigma_0$ as a compact annulus with constant mean curvature $H=1/2$ for which the following properties hold:
\begin{enumerate}
\item
Each coordinate $v$-curve, and in particular each component of $\parc \Sigma_0$, is a closed spherical curvature line of $\Sigma_0$.
 \item
$\Sigma_0$ is symmetric with respect to the $x_3=0$ plane, and with respect to $n$ equiangular vertical planes of $\R^3$ that intersect along a common vertical line. This line is actually the vertical line $L$ of Lemma \ref{lem:wen}.
 \item
The closed planar geodesic $\psi(0,v): [0,2n\sigma] \equiv \mathbb{S}^1 \to \{x_3 = 0\}$ has rotation index equal to $-1$.
\item
Both components of $\parc \Sigma_0$ lie in the same sphere of $\R^3$ of radius $R$, and $\Sigma_0$ intersects this sphere along $\parc \Sigma_0$ at a constant angle $\theta$. Here, $R,\theta$ are given by $$R^2= \frac{1+(z(u^*)-y(u^*))^2}{y^2(u^*)} , \hspace{1cm} \tan(\theta) = \frac{1}{z(u^*)-y(u^*)},$$ where $u^*$ is the value defined in Lemma \ref{ustar}.
 \item
If $\alpha > 1$, then $\Sigma_0$ is not rotational, and has a prismatic symmetry group of order $4n$.
 \item
If $\alpha = 1$, then $\Sigma_0$ is a compact piece of a nodoid and the line $L$ of Lemma \ref{lem:wen} is the rotation axis of $\Sigma_0$.
\end{enumerate}
\end{proposition}

\begin{proof} Assume first that $\alpha > 1$. The equality $\psi(u,v) = \psi(u,v + 2n\sigma)$ and the first three items follow from Corollary \ref{cor:sim}. Item (4) is a consequence of \eqref{radan} and the fact that $\hat{c}_3(u^*) = \hat{c}_3(-u^*) = 0$, which imply that the spheres $\mathbf{S}_{u^*}$ and $\mathbf{S}_{-u^*}$ coincide. 

%
%
%
%
%
%

To prove item (5), note that the curvature of the closed planar curve $\Gamma\equiv \Gamma(v)= \psi(0,v)$, which is given by 
\begin{equation}\label{forcur}
\kappa_2(0,v) = e^{-\omega(0,v)}\sinh \omega(0,v),
\end{equation} is not constant. Thus, $\Gamma$ is not a circle. Consequently, $\Sigma_0$ cannot be rotational. We next show that the symmetry group of $\Sigma_0$ is generated by the symmetries described in (2). Since this group corresponds to a prismatic subgroup or order $4n$ of the isometry group of $\R^3$, this would prove assertion (5).

Let $\Phi$ be an isometry of $\R^3$ that leaves $\Sigma_0$ invariant. Since $\Gamma$ is the only $v$-curvature line of $\Sigma_0$ that is planar, and the $u$-curvature lines of $\Sigma_0$ are not closed, we deduce that $\Phi(\Gamma)=\Gamma$. In this way, any symmetry $\Phi$ of $\Sigma_0$ is of the form $\Phi = (\Psi,\pm {\rm Id})$, where $\Psi$ is an isometry of the $\{x_3=0\}$ plane, with $\Psi(\Gamma)=\Gamma$. That is, $\Psi$ is an element of the planar symmetry group of $\Gamma$.

Besides, by item (2), $\Gamma$ is symmetric with respect to $n$ vertical planes $\{\Omega_1,\dots, \Omega_n\}$. Hence, we deduce that the planar symmetry group of $\Gamma$ is a dihedral group $D_{n'}$ that contains at least the reflections with respect to $n$ equiangular lines $\{L_1,\dots, L_n\}$, where $L_k=\Omega_k\cap \{x_3=0\}$. So, $n'\geq n$. If $n'>n$, there would exist an additional symmetry line $L'$ of $\Gamma$. At the intersection points of $\Gamma$ and $L'$, the curvature $\kappa_2$ of $\Gamma$, given by \eqref{forcur}, has a local maximum or minimum. But on the other hand, it was proved in Proposition \ref{pro:sim} that the curvature $\kappa_2$ between two consecutive lines $L_k, L_{k+1}$ is monotonic. This shows that this additional line $L'$ does not exist. Thus, $n=n'$, and the symmetry group of $\Sigma_0$ is generated by the isometries described in assertion (2).

Finally, we assume $\alpha = 1$. By our discussion in Section \ref{sec:nodoid}, the corresponding surface $\Sigma_0$ is the compact piece of nodoid obtained by rotating the profile curve $u \mapsto \psi(u,0)$, with $u\in [-u^*,u^*]$. The curves $v\mapsto \psi(u,v)$ are in this case horizontal circles, whose centers lie in the rotation axis of $\Sigma_0$. Moreover, the function $\omega(u,v)$ only depends on $u$, and $e^{\omega(0,v)} \equiv 1/\gamma$. 
Items (1) and (4) follow immediately from this discussion. Item (6) also follows after noting that if a horizontal circle lies in a sphere, then the vertical line that passes through the center of the circle also passes through the center of the sphere.


Now we will prove that $\psi(u,v) = \psi(u, v + 2n \sigma)$. Let $\xi(v):= e^{-\omega(0,v)}\Gamma'(v) \equiv \gamma \Gamma'(v)$ be the unit tangent vector of $\Gamma(v)$. According to \eqref{ecutheta}, if we denote by $\theta(v)$ the angle between $\xi(v)$ and a fixed direction of $\R^2$, the variation of $\theta(v)$ between $v = 0$ and $v = 2n\sigma$ is given by 
$$\Theta = \int_0^{2n\sigma}\sinh \omega(0,v) dv = n \sigma\left(\frac{1}{\gamma} - \gamma\right).$$
It follows then from \eqref{eq:pera1} and \eqref{eq:sigma1} together with ${\rm Per}(1,\beta,\gamma)=-1/n$ that $\Theta = -2\pi$. From here, we see that the values of the moving frame of $\psi$ at $(u,v)=(0,2n\sigma)$ agree with the ones at $(u,v)=(0,0)$, given by \eqref{inimovi}. Since $\omega$ does not depend on $v$, by the fundamental theorem of surface theory we obtain $\psi(u,v) = \psi(u, v + 2n \sigma)$. This proves item (3).



It just remains to prove item (2). This follows directly, since $\Sigma_0$ is symmetric with respect to the $x_3 = 0$ plane (see Remark \ref{re:radan}), and also with respect to any vertical plane that contains the axis $L$, by rotational symmetry. This completes the proof.
%
%
%
\end{proof}

%
%
%


In the next theorem we use the notation $u_1=u_1(\alfa,\beta,\gamma)$ of Proposition \ref{prosistem1}, as well as $\tau=\tau(\alfa,\beta,\gamma)$ of Proposition \ref{prosistem2}, and $u^*=u^*(\alfa,\beta,\gamma)$ of Lemma \ref{ustar}.

\begin{theorem}\label{thm:Gamma}
Let $\Upsilon:[0,1]\flecha \overline{\cW}$ be a continuous arc satisfying:
\begin{enumerate}
\item
$\Upsilon(0)=(\alfa_0,\alfa_0,\gamma_0)\in \cW.$
\item
$\Upsilon(\rho) \in \cW$ for every $\rho \in [0,1)$.
\item
$\Upsilon(1)= (\alfa_1,\beta_1,\gamma_1)\notin \cW$, where $\gamma_1>1$ and the values $(A_1,B_1,C_1)$ defined in terms of $\Upsilon(1)$ by \eqref{const2} satisfy
\begin{equation}\label{eq:igualdadlimite}
C_1^2 = \frac{(A_1 - B_1)^2}{4 A_1 B_1}.
\end{equation}
\end{enumerate}
Then there exists some $\rho^*\in (0,1)$ such that $u^*(\Upsilon(\rho^*))=\tau(\Upsilon(\rho^*))$.
\end{theorem}
\begin{proof}
We will prove that $u^*(\Upsilon(0)) < u_1(\Upsilon(0))= \tau(\Upsilon(0)) $ and that $\lim_{\rho \to 1^-}\tau(\Upsilon(\rho)) = 0 < u^*(\Upsilon(1))$, from where the existence of the desired value $\rho^* \in (0,1)$ follows. 
Note that, by Lemma \ref{ustar}, it holds $0 < u^*(\Upsilon(\rho)) < u_1(\Upsilon(\rho))$ for all $\rho \in [0,1]$. Also, note that $\tau(\Upsilon(\rho))$ is only defined for $\rho \in [0,1)$, since $\Upsilon(1) \notin \cW$.


To prove that $u_1(\Upsilon(0))=\tau(\Upsilon(0))$ we observe that, for $\Upsilon(0) = (\alpha_0,\alpha_0,\gamma_0)$, the solution $(y(u),z(u))$ to \eqref{system1}-\eqref{inicondi} satisfies $z(u) \equiv 0$ and $y(u) > 0$ for every $u\in (0,u_1)$. As a consequence, the first $\tau > 0$ for which $y(\tau)=z(\tau)(=0)$ must be $\tau = u_1$. That is, it holds $u_1(\Upsilon(0))=\tau(\Upsilon(0))$.

Now we claim that $\lim_{\rho \to 1^-}\tau(\Upsilon(\rho)) = 0$. Arguing by contradiction, assume that there is a sequence $(\rho_n)_{n}\to 1$ such that $\tau(\Upsilon(\rho_n)) > \varepsilon$ for some $\varepsilon > 0$. Let $(y_n(u),z_n(u))$ be the solutions to system \eqref{system1} with initial condition \eqref{inicondi} associated to the values $\Upsilon(\rho_n)$, and denote by $(y(u),z(u))$ the corresponding solution for the limit value $\Upsilon(1)$. 

The continuity of solutions to ODEs with respect to initial conditions and parameters shows that $(y_n(u), z_n(u))$ converges pointwise to $(y(u),z(u))$. By our hypothesis on $(\rho_n)_n$, we have $y_n(u)- z_n(u) > 0$ for $u \in (0,\ep)$. Thus, $y(u) \geq z(u)$ for every $u\in [0,\ep]$. However, Remark \ref{rem:ymenorz} and \eqref{eq:igualdadlimite} let us deduce that $y(u) - z(u) < 0$ for $u > 0$ sufficiently small, leading to a contradiction.
\end{proof}

\begin{remark}\label{rem:cambio}
If the curve $\Upsilon$ is real analytic, then so will be the function $$f(\rho) := u^*(\Upsilon(\rho)) - \tau(\Upsilon(\rho)): [0,1) \to \R.$$
Also, from the proof of Theorem \ref{thm:Gamma}, we have $f(0)<0$. Thus, since $f$ changes sign in $(0,1)$ and has isolated zeros (by analyticity), we can choose the point $\rho^*\in (0,1)$ of Theorem \ref{thm:Gamma} so that the conditions $f < 0$ in $(\rho^*-\varepsilon, \rho^*)$ and $f> 0$ in $(\rho^*, \rho^* + \varepsilon)$ hold, for some $\varepsilon>0$.
\end{remark}

By Proposition \ref{thm:Per1n} and Theorem \ref{thm:Gamma}, we have:

\begin{corollary}\label{cor:orthogonal}
Let $(\alpha,\beta,\gamma) \in \cW$ be the point $\Upsilon(\rho^*)$ given in Theorem \ref{thm:Gamma}. Let $\Sigma_0=\Sigma_0(\alfa,\beta,\gamma)$ be the compact CMC annulus defined in Proposition \ref{thm:Per1n}. Then, $\Sigma_0$ intersects orthogonally the sphere where $\partial \Sigma_0$ lies.
\end{corollary}

\subsection{Embedded free boundary CMC annuli in $\B^3$}

We now finally prove Theorem \ref{th:main}. 

Let $n \geq 2$, and define the analytic map $G:(0,\infty)\times (0,\8) \to \R^3$ given by
$$G(\alpha,\beta):= (\alpha,\beta,\gamma_c(\alpha,\beta)),$$
\noindent where $c:= -1/n$ and $\gamma_c$ was defined in Proposition \ref{pro:period}. So, by Proposition \ref{pro:period}, $G$ is a parametrization of the level set ${\rm Per}^{-1}(-1/n)$. Define next the curve $$\Upsilon(\beta):= G(1,\beta): [1,\infty) \to \R^3.$$
Under these conditions, we will show that there is some $\beta_1> 1$ such that $\Upsilon(\beta) \in \cW$ for all $\beta \in [1,\beta_1)$ and $\Upsilon(\beta_1)\notin \cW$ satisfies \eqref{eq:igualdadlimite}. To do so, we consider the auxiliary function 
$$L(\alpha,\beta,\gamma) := C^2 - \frac{(A - B)^2}{4AB},$$
\noindent where $\alfa,\beta,\gamma>0$ and $A, B, C$ are defined in terms of them by \eqref{const2}. From the definition of $L$, it follows that $L>0$ on $\cW$. According to \eqref{eq:Ckappa}, the composition $L(\Upsilon(\beta))$ can be written as
$$L(\Upsilon(\beta)) = \frac{c^2(B + 1)}{2(1-c^2)} - \frac{(B - 1)^2}{4B},$$
\noindent where $B =\frac{1}{2}\left( \beta + 1/\beta\right)$. Note that, since $c^2\leq 1/4$, we have 
$$L(\Upsilon(1)) =\frac{c^2}{1 - c^2}>0$$ 
and
$$\lim_{\beta \to \infty} L(\Upsilon(\beta)) = \lim_{B \to \infty} \frac{c^2(B + 1)}{2(1-c^2)} - \frac{(B - 1)^2}{4B} = -\infty.$$
As a consequence, there exists some $\beta_1 > 1$ verifying $L(\Upsilon(\beta_1)) = 0$ and $L(\Upsilon(\beta)) > 0$ for all $\beta\in [1,\beta_1)$. In other words, $\Upsilon(\beta) \in \cW$ for all $\beta \in [1,\beta_1)$ and $\Upsilon(\beta_1)=:(1,\beta_1,\gamma_1)\notin \cW$ satisfies \eqref{eq:igualdadlimite}. Applying Theorem \ref{thm:Gamma} we deduce the existence of some $\beta^* \in (1, \beta_1)$ such that $u^*(\Upsilon(\beta^*)) = \tau(\Upsilon(\beta^*))$.

Let $(1,\beta^*,\gamma^*) = \Upsilon(\beta^*) = G(1, \beta^*)$. We can assume that $\beta^*$ is in the conditions of Remark \ref{rem:cambio}. We now consider the function $$F(\alpha,\beta) =  u^*(G(\alpha,\beta))  - \tau(G(\alpha,\beta)).$$
\noindent Note that $F(1,\beta^*) = 0$. Also, $F$ is analytic around $(1,\beta^*)$, since the maps $u^*,\tau$ are analytic at $(1,\beta^*,\gamma^*)$, by Proposition \ref{prosistem1} and Lemma \ref{ustar}. This implies the existence of an analytic curve $$\cC(\mu) := (\alpha(\mu),\beta(\mu)):[0,\varepsilon) \to \R^2$$
such that $F(\cC(\mu)) \equiv 0$, with $\cC(0) = (1,\beta^*)$ and $\alpha(\mu) > 1$ for every $\mu\in (0,\ep)$. 

We now define $\Sigma(\mu)$, $\mu \in [0,\varepsilon)$, as the $H=1/2$ surface $\psi:\R^2\flecha \R^3$ of Definition \ref{def:Sigma} associated to the values $(\alfa,\beta,\gamma)=G(\cC(\mu))$. Also, as in Definition \ref{def:sigma0}, we define $\Sigma_0(\mu)$ as the restriction of the immersion $\psi(u,v)$ to $[-u^*(\mu),u^*(\mu)] \times \R$, where $u^*(\mu) := u^*(G(\cC(\mu)))$. We will also denote $\tau(\mu) := \tau(G(\cC(\mu)))$. From this construction we can deduce the following assertions:
\begin{enumerate}
\item \emph{Each $\Sigma_0(\mu)$ is a compact $H=1/2$ annulus that intersects orthogonally a sphere $\S_{\mu}$ of some radius $R(\mu)>0$ along $\parc \Sigma_0(\mu)$}. This follows from Proposition \ref{thm:Per1n} and Corollary \ref{cor:orthogonal}. 
\item \emph{$\Sigma(0)$ is the universal cover of a nodoid}. Indeed, this follows from our discussion in Section \ref{sec:nodoid}, and the fact that $\cC(0) = (1,\beta^*)$. In particular, $\Sigma_0(0)$ is a rotational annulus.
\item \emph{If $\mu>0$, then $\Sigma_0(\mu)$ has a prismatic symmetry group of order $4n$}. This is immediate, again by Proposition \ref{thm:Per1n}.
\item \emph{The family of compact annuli $\{\Sigma_0(\mu): \mu\in [0,\ep)\}$ is real analytic in terms of $\mu$}. This is immediate from the analytic nature of the construction.
\end{enumerate}

We will now prove that, maybe for a smaller $\ep>0$, every such annuli $\Sigma_0(\mu)$ with $\mu\in [0,\ep)$ is contained in the ball $B_{\mu}$ of $\R^3$ bounded by the sphere $\S_{\mu}$, and is embedded. 

To start, let $\psi_{\mu}(u,v):I_{\mu}\times \R\flecha \R^3$, where $I_{\mu}:=[-u^*(\mu),u^*(\mu)]$, denote our usual parametrization of the compact annulus $\Sigma_0(\mu)$. Recall that, by Proposition \ref{thm:Per1n}, $\psi_{\mu}$ is periodic in the $v$-direction with a fundamental period $2n\sigma$, where $\sigma=\sigma(\mu)>0$ denotes the analytic function $\sigma(\mu) := \sigma(G(\cC(\mu)))$; see Remark \ref{sigmana}. Therefore, we will view $\psi_{\mu}(u,v)$ as a parametrization of $\Sigma_0(\mu)$ defined on $I_\mu\times \S^1$ after the identification $(u,v)\sim (u,v+2n\sigma)$.


\begin{claim}\label{claim:nodoid}
$\psi_0 :I_0\times \S^1\flecha \Sigma(0)$ is an injective parametrization of the embedded free boundary nodoid ${\bf N}_{\gamma}\subset \Sigma(0)$ of Definition \ref{def:fbnodoid}, where $\gamma:= \gamma_c(1,\beta^*)$.
\end{claim}
\begin{proof}
Recall from Section \ref{sec:nodoid} that every curve $v\mapsto \psi_0(u_0,v)$ parametrizes a horizontal circle of the rotational surface $\Sigma(0)$. Since, by construction, the period map ${\rm Per}$ on $\Sigma(0)$ is equal to $-1/n$, it follows from Proposition \ref{thm:Per1n} that the rotation index of the planar geodesic $\psi(0,v):\S^1\flecha \{x_3=0\}$ is $-1$. Therefore, for any $u_0\in I_{0}$ we have that $v\in \S^1\mapsto \psi(u_0,v)$ is injective.

From now on, we will denote by $\Sigma$ the nodoid $\Sigma(0)$. Also, we denote $\Sigma_0 := \Sigma_0(0)$ and $u^*:= u^*(0) = \tau(0)$. Our goal will be to prove that $u^*$ coincides with the value $\delta \in (0,\overline{u})$ defined in Section \ref{sec:nodoid} associated to the nodoid ${\bf N}_{\gamma}$. This will prove that $\psi_0(I_0\times \S^1)=\Sigma_0={\bf N}_{\gamma}$, as stated.



%
%
%

Following the discussion in Section \ref{sec:nodoid}, the profile curve of the nodoid $\Sigma$ is given by $u \mapsto \psi(u,0)$. The curvature of this curve is $\kappa_1 = e^{-\omega}\cosh \omega$, where $\omega = \omega(u)$ is the solution to \eqref{hatode}. We note that $\kappa_1$ is strictly decreasing in the interval $u \in [0,\overline{u}]$ and reaches its minimum at $\overline{u}$; recall that $\overline{u}$ was also defined in Section \ref{sec:nodoid}. This can be used to show that $u^* < \overline{u}$, as follows. First, by the definition of $u^*$, for every $u \in (0,u^*]$ we have that $y(u) \geq z(u)$. Moreover, since $\beta^* > 1$, we have by \eqref{inicondi} and Proposition \ref{prosistem1} that $z(u)>0$ when restricted to that interval. As a consequence, by \eqref{omu},
$$\omega_u(u) = y(u)\cosh\omega + z(u)\sinh\omega \geq z(u)(\cosh\omega + \sinh\omega) = z(u) e^{\omega} > 0$$
\noindent for every $u \in (0,u^*]$. So, $\omega(u)$ is strictly increasing in an open interval $\mathcal{J}$ containing $(0,u^*]$, which means that $\kappa_1(u)$ is strictly decreasing in $\mathcal{J}$. Consequently, $u^*<\overline{u}$, since $\kappa_1(u)$ has a strict local minimum at $\bar{u}$.

Now we will study the vertical planar geodesic $\Gamma(u):= \psi(u,0)$. Since $u^* = \tau(0)$, we deduce that $y(u^*) = z(u^*)$, which implies by \eqref{gencen} that the tangent line to $\Gamma(u)$ at the point $\Gamma(u^*)$ passes through $\hat{c}(u^*)$. By definition of $u^*$ and Proposition \ref{thm:Per1n}, $\hat{c}(u^*)$ is the intersection point of the rotation axis of the nodoid with the plane $x_3 = 0$. We know that there is a unique value $\delta \in (0,\overline{u})$ satisfying this property (see Section \ref{sec:nodoid}), so $\delta$ must coincide with $u^*$. This completes the proof of the claim.
\end{proof}

As a consequence of Claim \ref{claim:nodoid}, the compact CMC annulus $\Sigma_0(0)={\bf N}_{\gamma}$ is embedded, and contained in the ball $B_0$ bounded by the sphere $\S_0$ where $\parc \Sigma_0(0)$ is contained. By the real analyticity of the family of compact annuli $\{\Sigma_0(\mu)\}_{\mu}$ we deduce that, for $\ep>0$ small enough and any $\mu\in [0,\ep)$, the annulus $\Sigma_0(\mu)$ is also embedded and contained in the ball $B_{\mu}$ bounded by $\S_{\mu}$.

Finally, for any $\mu\in [0,\ep)$, let $\Psi_\mu$ denote the homothety and translation of $\R^3$ that sends the sphere $B_{\mu}$ to the unit ball $\B^3$ of $\R^3$. Then, $\mathbb{A}_n(\mu):=\Psi_\mu(\Sigma_0(\mu))$ is an embedded free boundary CMC annulus in $\B^3$ with all the properties listed in Theorem \ref{th:main}. This completes the proof.

\section{Discussion of the examples}\label{sec:discussion}

The existence of the family $\mathbb{A}_n(\mu)$ of non-rotational embedded free boundary CMC annuli of Theorem \ref{th:main} can be seen as a bifurcation of certain embedded free boundary nodoids in $\B^3$. The bifurcation of complete nodoids in $\R^3$ was studied by Mazzeo and Pacard in \cite{MP}, where they showed the existence of bifurcation branches of the family of nodoids in $\R^3$ that give rise to complete, properly immersed, cylindrically bounded CMC annuli in $\R^3$ with a finite symmetry group. Our annuli $\mathbb{A}_n(\mu)$ come, after a homothety, from compact pieces of complete $H=1/2$ annuli with the same properties than the Mazzeo-Pacard examples. However, they are different, since they have different symmetry groups. We do not know if the Mazzeo-Pacard examples in \cite{MP} posses some compact portion that has free boundary in a ball of $\R^3$.

Our construction gives a sequence of mean curvature values $\{H_n\}_n$ for which the embedded free boundary nodoid in $\B^3$ with mean curvature $H_n>0$ bifurcates into our family of CMC annuli $\mathbb{A}_n(\mu)$. In that bifurcation, the free boundary condition is preserved, but the constant mean curvatures of the annuli $\mathbb{A}_n(\mu)$ are not equal to $H_n$, in general. The bifurcation values $H_n$ tend to $\8$ as $n\to \8$ and can be estimated numerically following the construction process described in this paper. We omit the details of this numerical estimation, since it is a bit involved. 

It remains open to show if for any $H>0$ there exists a non-rotational, embedded free boundary annulus in $\B^3$ with constant mean curvature $H$. We remark nonetheless that our construction does not work for $H=0$. Indeed, it was shown in \cite{FHM} that the only embedded free boundary minimal annulus in $\B^3$ foliated by spherical curvature lines is the critical catenoid. We also note that, if $n$ is even, the annuli $\mathbb{A}_n(\mu)$ are symmetric with respect to three orthogonal planes of $\R^3$. Thus, McGrath's characterization of the critical catenoid among free boundary minimal annuli in $\B^3$ (see \cite{M}) does not hold in the general CMC case.

It seems interesting, in the view of Theorem \ref{th:main}, to update Wente's uniqueness problem as follows: \emph{does every embedded free boundary CMC annulus in $\B^3$ have a family of spherical curvature lines?}

\def\refname{References}

\vskip 0.2cm

\noindent Alberto Cerezo

\noindent Departamento de Geometría y Topología \\ Universidad de Granada (Spain) \\ Departamento de Matemática Aplicada I \\ Universidad de Sevilla (Spain)

\noindent  e-mail: {\tt cerezocid@ugr.es}

\vskip 0.2cm

\noindent Isabel Fernández

\noindent Departamento de Matemática Aplicada I,\\ Instituto de Matemáticas IMUS \\ Universidad de Sevilla (Spain).

\noindent  e-mail: {\tt isafer@us.es}

\vskip 0.2cm

\noindent Pablo Mira

\noindent Departamento de Matemática Aplicada y Estadística,\\ Universidad Politécnica de Cartagena (Spain).

\noindent  e-mail: {\tt pablo.mira@upct.es}

\vskip 0.4cm

\noindent This research has been financially supported by Project PID2020-118137GB-I00 funded by MCIN/AEI /10.13039/501100011033.

\end{document}